\documentclass[11pt]{article}

\usepackage[utf8]{inputenc}
\usepackage[T1]{fontenc}
\usepackage{lmodern}
\usepackage{amsmath,amssymb,amsthm,mathtools}
\usepackage{graphicx}
\usepackage{booktabs}
\usepackage{enumitem}
\usepackage[margin=1in]{geometry}
\usepackage{hyperref}
\usepackage{xcolor}
\usepackage{cite}
\usepackage{caption}
\usepackage{authblk}

\newtheorem{theorem}{Theorem}[section]
\newtheorem{proposition}[theorem]{Proposition}

\newtheorem{corollary}[theorem]{Corollary}
\theoremstyle{definition}
\newtheorem{definition}[theorem]{Definition}
\newtheorem{assumption}[theorem]{Assumption}
\newtheorem{remark}[theorem]{Remark}

\newcommand{\R}{\mathbb{R}}
\newcommand{\eps}{\epsilon}
\DeclareMathOperator*{\argmin}{arg\,min}

\DeclareMathOperator{\dist}{dist}
\DeclareMathOperator{\diam}{diam}
\DeclareMathOperator{\proj}{proj}

\title{A Diagnostic Framework for Implementation Risk in Bilevel
Decision Problems: The Ambiguity Premium and the
Robustness--Efficiency Frontier}

\author{Jiguang Yu \\ \texttt{jyu678@bu.edu}}
\affil[1]{College of Engineering, Boston University, Boston, MA 02215, USA}

\begin{document}
\maketitle
  
\begin{abstract}
\noindent
Hierarchical decision problems are often modeled as bilevel programs
in which a leader commits to a policy and a follower responds
optimally. When the follower's optimal response is nonunique, or
when only near-optimal follower behavior can be verified, the same
leader decision may induce a range of upper-level outcomes. This
paper develops a diagnostic framework for quantifying that exposure.
For a leader decision $x$, we evaluate the optimistic and pessimistic
upper-level values over the $\eps$-optimal follower response set
$S_\eps(x)$ and use their difference,
\[
  \Delta_\eps(x):=\psi_\eps^p(x)-\psi_\eps^o(x),
\]
as an ambiguity premium. The premium itself is classical in the
optimistic--pessimistic bilevel distinction; the contribution here is
to make it operational as an implementation-risk diagnostic. We
establish a diameter bound
$\Delta_\eps(x)\le L_F(x)\,\diam(S_\eps(x))$ and an $O(\sqrt{\eps})$
estimate under quadratic lower-level growth. We then organize
existing bilevel--GNEP reformulations by their computational roles
and propose a screening workflow that reports, for each candidate
policy, nominal value, ambiguity exposure, and a first-order
residual. Two stylized case studies---a parallel-link Stackelberg
pricing problem and a convex generation-planning model with
diversification constraints---show how the resulting
robustness--efficiency frontier can identify policies that are
nominally attractive but sensitive to near-optimal follower
responses.
\end{abstract}


\section{Introduction}\label{sec:intro}
Bilevel optimization is a natural modeling framework for hierarchical decision problems in which one agent commits to a policy and a second agent reacts optimally~\cite{colson2007overview,liu2025bidirectional,wang2026algebraic,wang2025analysis,
dempe2002foundations,zemkohoo2020bilevel,kleinert2021survey,luo1996mathematical}. 
Such structures arise in Stackelberg games and network pricing~\cite{labbe1998bilevel,brotcorne2001bilevel,dempe2013bilevel,
gilbert2015numerical}, regulation~\cite{bard2013practical,wan2014estimation}, infrastructure planning
and energy systems~\cite{gabriel2012complementarity,wang2025analysis1,liang2025global,wogrin2020applications},
interdiction and security~\cite{smith2020survey,wang2026introduction1,caprara2016bilevel,israeli2002shortest},
hyperparameter and machine-learning pipelines~\cite{franceschi2018bilevel,wang2026damage,liu2021investigating}, and many
other areas. 
Their appeal is also their main difficulty: the leader's feasible region is defined implicitly through the follower's solution map, and the upper-level outcome may depend critically on how lower-level optimal responses are selected. This is especially acute when the follower admits multiple optima,
giving rise to the classical distinction between optimistic and pessimistic bilevel models~\cite{loridan1996weak,dempe2002foundations,yu2026pattern,wiesemann2013pessimistic,wang2025multi,beck2023survey,gao2022rolling,besanccon2019near,wang2026lecture,wang2026first,wang2026introduction}.

\paragraph{Motivation.}
From a decision-maker's perspective, follower nonuniqueness is not merely a technical nuisance. Two follower responses that are equally acceptable from the follower's point of view can induce materially
different upper-level outcomes. 
Moreover, in many real systems the leader can only verify that a follower's response is approximately optimal: numerical tolerances, bounded rationality, contractual flexibility, or simply the existence of
multiple equally good operating points all create a neighborhood of acceptable follower responses rather than a single point. Denoting this $\eps$-optimal response set by $S_\eps(x)$, the relevant
quantity for a planner is not only the nominal leader value, but
also the width of the leader's \emph{value interval}
$[\psi_\eps^o(x),\psi_\eps^p(x)]$ generated by $S_\eps(x)$.

\paragraph{Contribution.}
The optimistic--pessimistic value gap is a classical object in
bilevel optimization. Our contribution is not to introduce this
gap, but to use it as the organizing quantity in a decision
diagnostic for implementation risk. Specifically, we make four
contributions.

First, we define a candidate-level diagnostic triple consisting of
nominal optimistic value, ambiguity exposure, and a first-order
residual, and we summarize the first two components through a
robustness--efficiency frontier. Second, we connect the ambiguity
premium to response-set geometry by proving a Lipschitz diameter
bound and an $O(\sqrt{\eps})$ estimate under quadratic lower-level
growth. Third, we organize existing bilevel--GNEP
reformulations~\cite{facchinei2010generalized,lampariello2017bridge,
lampariello2020numerically,lampariello2019standard} according to their
computational roles: global optimistic reformulation,
stationarity-oriented computation, and pessimistic
lower-equilibrium evaluation. Fourth, we demonstrate the workflow
on two stylized examples chosen to isolate different sources of
ambiguity: exact follower multiplicity in a Stackelberg pricing
problem and near-optimal redispatch in a strictly convex planning
problem.

The analytical role of the paper is to collect simple stability
consequences of the diagnostic and make them operational for
screening candidate leader decisions. The computational workflow
uses Nikaido--Isoda-type penalization~\cite{nikaido1955note,fukushima1992equivalent}
for pessimistic evaluation and a proximal alternating scheme for
optimistic candidate generation, with a Fischer--Burmeister
residual~\cite{fischer1992special,facchinei2003finite} replacing ad hoc
residuals.

\paragraph{Scope.}
The framework is diagnostic rather than predictive. We do not
claim new bilevel--GNEP equivalence theorems, nor do we claim
global optimality for all numerically reported pessimistic
portfolios. The reformulations used below are drawn from the
existing GNEP and bilevel literature. The case studies are
stylized computational experiments designed to show how the
diagnostic behaves, not empirical forecasts or policy
recommendations.

\paragraph{Positioning relative to recent literature.}
Recent work has greatly expanded the scope of bilevel analysis
under uncertainty and
equilibrium~\cite{beck2023survey,besanccon2019near,
buchheim2022stochastic,burtscheidt2020risk}, has produced improved
global solution methods for optimistic and pessimistic
problems~\cite{kleniati2015generalization,djelassi2021recent,liu2018pessimistic,
mitsos2008global,tsoukalas2009global}, and has clarified the role of
generalized Nash equilibria in hierarchical
settings~\cite{facchinei2007generalized,dreves2011solution,
kulkarni2012variational}. Our framework is complementary to this
computational work: it takes the existence of such solvers as a
given and adds a decision-analytic overlay that evaluates a
candidate decision along three dimensions at once.

\paragraph{Organization.} Section~\ref{sec:problem} fixes notation,
consolidates definitions of the lower-level value function and
solution map, and defines the decision diagnostics.
Section~\ref{sec:theory} presents the analytical results linking
the ambiguity premium to diameter and quadratic-growth moduli.
Section~\ref{sec:reform} reviews the three GNEP reformulations
used in the workflow and states their computational roles.
Section~\ref{sec:algo} describes the computational workflow.
Sections~\ref{sec:case1}--\ref{sec:case2} present the two case
studies. Section~\ref{sec:managerial} discusses managerial
implications with softened claims.
Section~\ref{sec:conclusion} concludes. Appendix~\ref{sec:proofs}
contains proofs; Appendix~\ref{sec:repro} documents
reproducibility.

\section{Problem Class, Value Functions, and Diagnostics}
\label{sec:problem}
We consider a parametric bilevel setting with leader decision
$x\in X\subset\R^n$ and follower decision $y\in Y\subset\R^m$. We
make the following standing assumption throughout the paper, which
is mild and standard~\cite{bonnans2013perturbation,zemkohoo2020bilevel}.

\begin{assumption}[Standing assumptions]\label{ass:standing}
$X\subset\R^n$ and $Y\subset\R^m$ are nonempty compact sets.
$F:X\times Y\to\R$ and $f:X\times Y\to\R$ are continuous.
\end{assumption}

\noindent
Under Assumption~\ref{ass:standing}, the lower-level value function
and exact solution map
\begin{equation}\label{eq:phi-S}
  \phi(x):=\min_{v\in Y} f(x,v),
  \qquad
  S(x):=\argmin_{v\in Y} f(x,v)=\{y\in Y:f(x,y)=\phi(x)\}
\end{equation}
are well defined, $\phi$ is continuous on $X$, and $S$ is
nonempty-valued and upper semicontinuous~\cite{bonnans2013perturbation}. The
$\eps$-optimal response set for $\eps\ge 0$ is
\begin{equation}\label{eq:Seps}
  S_\eps(x):=\{y\in Y : f(x,y)\le \phi(x)+\eps\}.
\end{equation}
By construction
$S(x)=S_0(x)\subseteq S_{\eps_1}(x)\subseteq S_{\eps_2}(x)\subseteq Y$
for $0\le\eps_1\le\eps_2$.

The optimistic and pessimistic leader value functions at tolerance
$\eps\ge 0$ are
\begin{align}
  \psi_\eps^o(x)
  &:=\min_{y\in S_\eps(x)} F(x,y),
  \label{eq:psi-o}\\
  \psi_\eps^p(x)
  &:=\max_{y\in S_\eps(x)} F(x,y).
  \label{eq:psi-p}
\end{align}
Both the minimum and the maximum are attained under
Assumption~\ref{ass:standing}, because $S_\eps(x)$ is a nonempty
compact subset of $Y$ and $F(x,\cdot)$ is continuous. The
ambiguity premium is
\begin{equation}\label{eq:premium-def}
  \Delta_\eps(x):=\psi_\eps^p(x)-\psi_\eps^o(x)\ge 0.
\end{equation}

\subsection{Decision diagnostics}
\label{sec:diagnostics}
We assemble the quantities used to compare candidate leader
decisions into the following decision-diagnostic set. 

\begin{definition}[Decision diagnostics]\label{def:diagnostics}
For $x\in X$, $y\in Y$, and $\eps\ge 0$:
\begin{enumerate}[leftmargin=*,nosep]
\item \emph{Ambiguity premium:}
$\Delta_\eps(x)=\psi_\eps^p(x)-\psi_\eps^o(x)$.
\item \emph{Normalized ambiguity ratio:}
$\rho_\eps(x):=\Delta_\eps(x)/(1+|\psi_\eps^o(x)|)$.
\item \emph{Tolerance violation residual:}
\begin{equation}\label{eq:rLL}
  r_\mathrm{LL}^\eps(x,y):=\max\{0,\ f(x,y)-\phi(x)-\eps\}.
\end{equation}
\item \emph{Stationarity residual.} The definition of a usable
first-order residual requires some care because $\phi$ is
continuous but generally nondifferentiable. We therefore define
the stationarity residual through the GNEP reformulation
\eqref{eq:G1} of Section~\ref{sec:reform}, which avoids
$\nabla_x\phi$. Concretely, at a triple $(x,y,v)\in X\times Y\times Y$
with multiplier $\lambda\ge 0$ on the coupling constraint
$h(x,y,v):=f(x,y)-f(x,v)-\eps\le 0$, we use only first derivatives
of $f$ and $F$, which are well defined whenever $F$ and $f$ are
continuously differentiable on a neighborhood of the candidate
point. 
\end{enumerate}
\end{definition}

\begin{remark}[Assumptions where they are used]\label{rem:assum-map}
The standing assumptions (compactness and continuity) are used for
Definition~\ref{def:diagnostics}(a)--(c). Definition
\ref{def:diagnostics}(d) is well defined only under the additional
smoothness assumption that $F$ and $f$ are continuously
differentiable in $(x,y)$ on a neighborhood of the candidate
point; we make this assumption explicit when the stationarity
residual is used, and not globally.
\end{remark}

We also use the Pareto set in the plane
$(\psi_\eps^o,\Delta_\eps)$:
\begin{equation}\label{eq:pareto}
  \mathcal P_\eps
  :=\bigl\{(\psi_\eps^o(x),\Delta_\eps(x)):x\in X,\ x\text{ is not dominated}\bigr\},
\end{equation}
where $(u,\delta)$ dominates $(u',\delta')$ iff $u\le u'$,
$\delta\le\delta'$, and at least one inequality is strict. We call
$\mathcal P_\eps$ the robustness--efficiency frontier. Its
practical utility is that it compresses the joint evaluation of
nominal value and implementation exposure into a two-dimensional
picture without throwing away either dimension.

\section{Structural Properties of the Ambiguity Premium}
\label{sec:theory}
This section contains the main analytical contribution of the
paper. The results are not deep, but they link the ambiguity
premium to standard stability quantities and therefore give the
diagnostic analytical content beyond its definitional form.

Throughout, let $\diam(A):=\sup_{y,y'\in A}\|y-y'\|$ denote the
diameter of a nonempty bounded set $A\subseteq\R^m$.

\subsection{Attainment and interval width}
\label{sec:theory-attain}
\begin{proposition}[Attainment and interval width]\label{prop:attain}
Under Assumption~\ref{ass:standing}, for every $x\in X$ and
$\eps\ge 0$ the set $S_\eps(x)$ is nonempty and compact; both
$\psi_\eps^o(x)$ and $\psi_\eps^p(x)$ are attained; and
\begin{equation}\label{eq:interval}
  \Delta_\eps(x)
  =\diam_\R\bigl(F(x,S_\eps(x))\bigr)
  =\max_{y\in S_\eps(x)}F(x,y)-\min_{y\in S_\eps(x)}F(x,y),
\end{equation}
where $\diam_\R(A):=\sup A-\inf A$ is the one-dimensional diameter
of a bounded nonempty $A\subseteq\R$.
\end{proposition}

\begin{proof}
Nonemptiness of $S(x)\subseteq S_\eps(x)$ and compactness follow
from continuity of $f(x,\cdot)$ on the compact set $Y$. The
image $F(x,S_\eps(x))$ is then a nonempty compact subset of $\R$,
so its supremum and infimum are attained and
$\Delta_\eps(x)=\sup F(x,S_\eps(x))-\inf F(x,S_\eps(x))$.
\end{proof}

\subsection{Monotonicity in follower tolerance}
\begin{proposition}[Monotonicity]\label{prop:monotone}
Fix $x\in X$. For $0\le\eps_1\le\eps_2$:
$S_{\eps_1}(x)\subseteq S_{\eps_2}(x)$,
$\psi_{\eps_2}^o(x)\le\psi_{\eps_1}^o(x)$,
$\psi_{\eps_2}^p(x)\ge\psi_{\eps_1}^p(x)$, and
$\Delta_{\eps_2}(x)\ge\Delta_{\eps_1}(x)$.
\end{proposition}

\begin{proof}
If $y\in S_{\eps_1}(x)$, then $f(x,y)\le\phi(x)+\eps_1\le\phi(x)+\eps_2$,
so $y\in S_{\eps_2}(x)$. By Proposition~\ref{prop:attain} the
minima and maxima are attained, so minimization over a larger set
cannot increase the minimum and maximization cannot decrease the
maximum. The stated inequalities follow.
\end{proof}

\subsection{A diameter-based Lipschitz bound}
The following bound is elementary but useful as a diagnostic
device: it connects the ambiguity premium to a response-set
geometry quantity (the diameter of the $\eps$-response set) and a
leader-side Lipschitz modulus. We present it as a diagnostic
bound rather than a deep structural theorem.

\begin{theorem}[Diameter-based bound]\label{thm:diam-bound}
Assume $F(x,\cdot)$ is $L_F(x)$-Lipschitz continuous on $Y$, i.e.
\begin{equation}\label{eq:LF}
  |F(x,y)-F(x,y')|\le L_F(x)\|y-y'\|
  \qquad\forall y,y'\in Y.
\end{equation}
Then for every $\eps\ge 0$
\begin{equation}\label{eq:diam-bound}
  \Delta_\eps(x)\le L_F(x)\cdot\diam(S_\eps(x)).
\end{equation}
\end{theorem}

\begin{proof}
By Proposition~\ref{prop:attain} there exist
$y^+,y^-\in S_\eps(x)$ with $F(x,y^+)=\psi_\eps^p(x)$ and
$F(x,y^-)=\psi_\eps^o(x)$. Then
\[
  \Delta_\eps(x)=F(x,y^+)-F(x,y^-)
  \le L_F(x)\|y^+-y^-\|
  \le L_F(x)\diam(S_\eps(x)).
\]
\end{proof}

Theorem~\ref{thm:diam-bound} has two useful consequences. First,
one can upper-bound implementation risk with any tractable
bound on the diameter of the $\eps$-response set, regardless of the
specific leader objective. Second, the modulus $L_F(x)$ is
explicitly computable in most applications (e.g., $L_F(x)$ can be
taken as a bound on $\|\nabla_y F(x,\cdot)\|$ on $Y$).

\subsection{Rate under quadratic growth}
The diameter bound is most informative when $S_\eps(x)$ shrinks
with $\eps$. A standard sufficient condition for this is
quadratic growth of the lower-level objective around its optimal
set~\cite{bonnans2013perturbation,drusvyatskiy2018error,fabian2010error}.

\begin{assumption}[Quadratic lower-level growth]\label{ass:qg}
There exists $\mu(x)>0$ such that
\begin{equation}\label{eq:qg}
  f(x,y)\ge \phi(x)+\mu(x)\,\dist\bigl(y,S(x)\bigr)^2
  \qquad \forall y\in Y.
\end{equation}
\end{assumption}

If $Y$ is convex and $f(x,\cdot)$ is $2\mu(x)$-strongly convex
on $Y$, then $S(x)$ is a singleton and Assumption~\ref{ass:qg}
holds with the displayed $\mu(x)$. More generally, quadratic
growth is a standard ingredient of error-bound
theory~\cite{fabian2010error,drusvyatskiy2018error,luo1994extension}
and holds under a variety of second-order sufficient conditions
that do not require convexity of the lower-level objective.

\begin{corollary}[$O(\sqrt\eps)$ rate]\label{cor:sqrt-rate}
Under Assumptions~\ref{ass:standing}, \eqref{eq:LF},
and~\ref{ass:qg},
\begin{equation}\label{eq:diam-sqrt}
  \diam\bigl(S_\eps(x)\bigr)\le \diam(S(x))+2\sqrt{\eps/\mu(x)},
\end{equation}
and
\begin{equation}\label{eq:delta-sqrt}
  \Delta_\eps(x)\le L_F(x)\Bigl(\diam(S(x))+2\sqrt{\eps/\mu(x)}\Bigr).
\end{equation}
If in addition $S(x)=\{y^\star(x)\}$ is a singleton, then
$\Delta_\eps(x)\le 2L_F(x)\sqrt{\eps/\mu(x)}$.
\end{corollary}

\begin{proof}
For any $y\in S_\eps(x)$, quadratic growth
gives $\mu(x)\dist(y,S(x))^2\le f(x,y)-\phi(x)\le\eps$, so
$\dist(y,S(x))\le\sqrt{\eps/\mu(x)}$. Hence
$S_\eps(x)\subseteq S(x)+B(0,\sqrt{\eps/\mu(x)})$. The diameter
bound $\diam(S_\eps(x))\le\diam(S(x))+2\sqrt{\eps/\mu(x)}$ follows
from the triangle inequality. Combining
with~\eqref{eq:diam-bound} yields~\eqref{eq:delta-sqrt}. The
singleton case gives $\diam(S(x))=0$.
\end{proof}

\begin{remark}[Local versus global quadratic growth]\label{rem:qg-local}
Assumption~\ref{ass:qg} is stated globally on $Y$. For the
$O(\sqrt\eps)$ rate as $\eps\downarrow 0$, only a local form is
needed: it suffices that the inequality~\eqref{eq:qg} holds on a
neighborhood of $S(x)$ in $Y$, with the constant $\mu(x)$ adapted
to that neighborhood. The global form is mathematically safer but
strictly stronger than what is required.
\end{remark}

\begin{remark}[Interpretation]\label{rem:sqrt}
Corollary~\ref{cor:sqrt-rate} explains why the ambiguity premium
in the strictly convex case study of
Section~\ref{sec:case2} scales roughly like $\sqrt\eps$ and
remains positive even when the exact follower optimum is unique.
It also identifies the two modeling levers a designer has for
controlling implementation risk: reducing $L_F(x)$ (flatten the
leader objective along likely redispatch directions) or
increasing $\mu(x)$ (sharpen the lower-level objective).
\end{remark}

\subsection{Semicontinuity and continuity under response-set
continuity}
The following semicontinuity result is used below to justify the
use of sampling-based frontier approximations. Note that the
optimistic and pessimistic value functions have opposite
semicontinuity directions in general; full continuity of both
requires continuity of the correspondence
$(x,\eps)\mapsto S_\eps(x)$ in the sense required by Berge's
maximum theorem.

\begin{proposition}[Semicontinuity and continuity under response-set
continuity]\label{prop:continuity}
Under Assumption~\ref{ass:standing}, the correspondence
\[
(x,\eps)\longmapsto S_\eps(x):=\{y\in Y:\ f(x,y)\le \phi(x)+\eps\}
\]
has nonempty compact values and is upper semicontinuous on
$X\times[0,\infty)$. Furthermore, the marginal functions
$\psi_\eps^p$ and $\Delta_\eps$ are upper semicontinuous, and
$\psi_\eps^o$ is lower semicontinuous, in $(x,\eps)$.

If, in addition, Assumption~\ref{ass:qg} holds and there exists a
neighborhood $U$ of $x_0$ and a constant $\underline\mu>0$ such
that
\[
\mu(x)\ge \underline\mu\qquad \text{for all }x\in U\cap X,
\]
then $S_\eps(x)$ is lower semicontinuous at $(x_0,0)$.

Moreover, if we additionally assume that $Y$ is convex and
$f(x,\cdot)$ is convex on $Y$ for all $x$ near $x_0$, then
$S_\eps(x)$ is lower semicontinuous at $(x_0,\eps_0)$ for any
$\eps_0>0$.

Consequently, under these combined conditions,
$\psi_\eps^o$, $\psi_\eps^p$, and $\Delta_\eps$ are continuous at
$(x_0,\eps_0)$ for all $\eps_0\ge 0$.
\end{proposition}

\subsection{A formal local expansion under stronger regularity}
\label{sec:formal-expansion}
The bound in Corollary~\ref{cor:sqrt-rate} is order-of-magnitude.
Under stronger regularity one can describe the leading-order
behavior of $\Delta_\eps(x)$ in $\eps$ more precisely. We present
the discussion as a formal local expansion in the spirit of
classical second-order sensitivity
analysis~\cite{bonnans2013perturbation,shapiro1988sensitivity,fiacco1983introduction} rather than as a general theorem; the
relevant structural conditions can fail at active constraints
or under degeneracy.

Throughout this subsection we assume the singleton case
$S(x)=\{y^\star(x)\}$, $f$ twice continuously differentiable on a
neighborhood of $y^\star(x)$ in $Y$, and a second-order sufficient
condition at $y^\star(x)$. Let $\mathcal T_Y(y^\star(x))$ denote
the tangent cone to $Y$ at $y^\star(x)$. Following standard
practice~\cite{bonnans2013perturbation}, define the critical cone of zero
first-order increase
\begin{equation}\label{eq:crit-cone}
  \mathcal C(x)
  :=\{d\in\mathcal T_Y(y^\star(x)):
    \nabla_y f(x,y^\star(x))^\top d=0\}.
\end{equation}
Under standard constraint qualifications and local uniqueness, the
dominant $\sqrt\eps$ part of the sublevel-set displacement at
$y^\star(x)$ is governed by the trust-region-type set
\begin{equation}\label{eq:Dx}
  D(x)
  :=\bigl\{d\in\mathcal C(x):
    \tfrac12\,d^\top\nabla_{yy}^2 f(x,y^\star(x))\,d\le 1\bigr\}.
\end{equation}
Along noncritical feasible directions, $f$ has nonzero first-order
growth, so the sublevel displacement is at most $O(\eps)$ rather
than $O(\sqrt\eps)$; the $\sqrt\eps$ contribution therefore comes
from $\mathcal C(x)$.

Formally, when the first-order variation of $F$ along the critical
cone is nonzero, the leading contribution to $\Delta_\eps(x)$ takes
the form
\begin{equation}\label{eq:directional}
  \Delta_\eps(x)
  =\sqrt{\eps}\biggl[
    \sup_{d\in D(x)}\nabla_y F(x,y^\star(x))^\top d
    -\inf_{d\in D(x)}\nabla_y F(x,y^\star(x))^\top d
  \biggr]
  +o(\sqrt\eps).
\end{equation}
This expression involves two trust-region-type subproblems and is
computable from local first- and second-order data of $F$ and $f$
at $y^\star(x)$. When $D(x)$ is centrally symmetric---for instance,
when no inequality constraint of the lower-level problem is active
at $y^\star(x)$, so that $\mathcal C(x)$ is a linear subspace and
$D(x)$ is an ellipsoid---\eqref{eq:directional} reduces to
\begin{equation}\label{eq:directional-sym}
  \Delta_\eps(x)=2\sqrt{\eps}\,
  \sup_{d\in D(x)}\nabla_y F(x,y^\star(x))^\top d+o(\sqrt\eps).
\end{equation}
In the presence of active inequality constraints, $\mathcal C(x)$
need not be a subspace and $D(x)$ need not be symmetric, in which
case the general form~\eqref{eq:directional} is required. We do
not use \eqref{eq:directional}--\eqref{eq:directional-sym}
elsewhere in the paper; they are recorded here to indicate that
the order-of-magnitude bound of Corollary~\ref{cor:sqrt-rate} can
in principle be sharpened.

\section{Equilibrium Reformulations: Background and Roles}
\label{sec:reform}
The preceding section treats $\Delta_\eps(x)$ as a mathematical
object attached to a fixed leader decision. To use it as a
practical screening diagnostic, one must compute or approximate
$\psi_\eps^o(x)$ and $\psi_\eps^p(x)$ repeatedly across
candidate decisions. This section recalls three existing
bilevel--GNEP reformulations and assigns each a computational
role in that task. The purpose is not to establish new
equivalence results, but to make explicit which reformulation
supports which part of the diagnostic workflow.

The results used here are taken from~\cite{lampariello2017bridge,
lampariello2020numerically,lampariello2019standard};
see also~\cite{facchinei2010generalized,harker1991generalized,pang2005quasi,kulkarni2012variational,
dreves2011solution} for the GNEP machinery.

Recall the coupling constraint
$g_\eps(x,y)=f(x,y)-\phi(x)-\eps$.

\paragraph{Standing smoothness assumptions for this section.}
For all KKT displays in this section, assume in addition that $X$
and $Y$ are closed convex sets and that $F$ and $f$ are
continuously differentiable on a neighborhood of every candidate
point. If $X$ or $Y$ is nonconvex, the same formal expressions
should be interpreted with an appropriate limiting normal cone
(Mordukhovich or Clarke) and the corresponding constraint
qualifications; we do not develop the nonconvex case further.

\subsection{Optimistic model: global-solution reformulation (G1)}
\label{sec:G1}
The two-player GNEP
\begin{equation}\tag{G1}\label{eq:G1}
\begin{aligned}
  &\text{Player 1:}\quad \min_{x,y}\ F(x,y)
  \quad \text{s.t. } x\in X,\ y\in Y,\ f(x,y)\le f(x,v)+\eps,\\
  &\text{Player 2:}\quad \min_{v}\ f(x,v)
  \quad \text{s.t. } v\in Y,
\end{aligned}
\end{equation}
has the property that, at any equilibrium
$(x^\star,y^\star,v^\star)$, player 2's optimality gives
$f(x^\star,v^\star)=\phi(x^\star)$ and hence
$f(x^\star,y^\star)\le\phi(x^\star)+\eps$, so
$(x^\star,y^\star)$ is feasible for the optimistic bilevel
problem. The equivalence between \eqref{eq:G1} and the optimistic
bilevel problem is not, however, a formal consequence of the
displayed constraints alone. Under player-1 deviations the
constraint $f(x,y)\le f(x,v)+\eps$ is evaluated at fixed $v$, so
it does not by itself encode $f(x,v)=\phi(x)$. Equivalence with
global optimistic solutions relies on the structural assumptions
stated in~\cite{lampariello2017bridge}, which ensure that
equilibria of the game select globally valid optimistic bilevel
solutions. We use \eqref{eq:G1} only as a conceptual device for
defining a stationarity residual and for organizing the
computational workflow; all equivalence claims are inherited from
the cited results.

\paragraph{KKT system.}
Let $\lambda\ge 0$ be player 1's multiplier on the coupling
constraint
$h(x,y,v):=f(x,y)-f(x,v)-\eps\le 0$, and note that player 1
controls $(x,y)$ while player 2 controls $v$. No coupling
multiplier appears in player 2's problem, because player 2's
feasible set $Y$ does not depend on $(x,y)$. The KKT system of
the GNEP is
\begin{align}
  0&\in\nabla_x F(x,y)+\lambda\bigl(\nabla_x f(x,y)-\nabla_x f(x,v)\bigr)+N_X(x),
  \label{eq:G1-kkt-x}\\
  0&\in\nabla_y F(x,y)+\lambda\,\nabla_y f(x,y)+N_Y(y),
  \label{eq:G1-kkt-y}\\
  0&\in \nabla_v f(x,v)+N_Y(v),
  \label{eq:G1-kkt-v-fixed}\\
  0&\le\lambda\ \perp\ f(x,y)-f(x,v)-\eps\le 0.
  \label{eq:G1-kkt-comp}
\end{align}
We need to emphasize that in~\eqref{eq:G1-kkt-v-fixed} player 2 minimizes
$f(x,\cdot)$ on $Y$ and sees no coupling constraint; this is the
source of the identity $f(x^\star,v^\star)=\phi(x^\star)$ at
equilibrium.

\subsection{Optimistic model: convexified stationarity reformulation (G2)}
\label{sec:G2}
Replacing $f(x,v)$ in the coupling constraint by its first-order
lower approximation in the $x$-variable yields
\begin{equation}\tag{G2}\label{eq:G2}
\begin{aligned}
  &\text{Player 1:}\quad \min_{x,y}\ F(x,y)
  \quad \text{s.t. } x\in X,\ y\in Y,\
  f(x,y)\le f(u,v)+\nabla_1 f(u,v)^\top(x-u)+\eps,\\
  &\text{Player 2:}\quad \min_{u,v}\ f(x,v)
  \quad \text{s.t. } u=x,\ v\in Y.
\end{aligned}
\end{equation}
Under the structural assumptions stated
in~\cite{lampariello2020numerically}---which include convexity of
$f(\cdot,y)$ in the leader variable and appropriate constraint
qualifications---solutions of \eqref{eq:G2} correspond to
stationary points of the optimistic bilevel problem and vice
versa. We do not reproduce these conditions here, and we use
\eqref{eq:G2} computationally as a stationarity-targeting
heuristic.

\subsection{Pessimistic model: exact lower-equilibrium reformulation (MF)}
\label{sec:MF}
For each $x\in X$ the pessimistic leader value is generated
exactly by the lower-level GNEP
\begin{equation}\tag{MF}\label{eq:MF}
\begin{aligned}
  &\text{Player 1:}\quad \min_{y}\ -F(x,y)
  \quad \text{s.t. } y\in Y,\ f(x,y)\le f(x,v)+\eps,\\
  &\text{Player 2:}\quad \min_{v}\ f(x,v)
  \quad \text{s.t. } v\in Y.
\end{aligned}
\end{equation}
At equilibrium, $f(x,v^\star)=\phi(x)$, and
$F(x,y^\star)=\psi_\eps^p(x)$. The pessimistic bilevel problem is
therefore $\min_{x\in X}\psi_\eps^p(x)$
\cite{lampariello2019standard}.

\paragraph{KKT system.}
With multiplier $\lambda\ge 0$ on player 1's coupling constraint
and nothing on player 2 (which sees only $v\in Y$):
\begin{align}
  0&\in -\nabla_y F(x,y)+\lambda\,\nabla_y f(x,y)+N_Y(y),
  \label{eq:MF-kkt-y}\\
  0&\in \nabla_v f(x,v)+N_Y(v),
  \label{eq:MF-kkt-v-fixed}\\
  0&\le\lambda\ \perp\ f(x,y)-f(x,v)-\eps\le 0.
  \label{eq:MF-kkt-comp}
\end{align}
Note that in~\eqref{eq:MF-kkt-y}, player 1 sees $x$ as fixed, so
only $\nabla_y f(x,y)$ appears in the coupling-constraint
gradient; no $-\nabla_y f(x,v)$ term is present because $v$
belongs to player 2, not player 1.

\paragraph{Nikaido--Isoda gap function.}
An alternative characterization useful for
computation~\cite{nikaido1955note,fukushima1992equivalent,harker1991generalized,von2009optimization}
is the Nikaido--Isoda gap function
\begin{equation}\label{eq:NI}
\begin{aligned}
  \mathcal N_x(y,v)
  :=\;& \sup_{\hat y\in Y:\,f(x,\hat y)\le f(x,v)+\eps}
      \bigl[F(x,\hat y)-F(x,y)\bigr]\\
      &+\sup_{\hat v\in Y}\bigl[f(x,v)-f(x,\hat v)\bigr].
\end{aligned}
\end{equation}
Then $(y,v)$ solves \eqref{eq:MF} iff $\mathcal N_x(y,v)=0$, and
$F(x,y)=\psi_\eps^p(x)$. We use this formulation computationally
in Section~\ref{sec:algo} to give teeth to the claim that the
pessimistic reformulation is doing real work.

We summarize the distinct computational roles that the three reformulations play in Table~\ref{tab:reform-roles}.

\begin{table}[ht]
\centering
\caption{Computational roles of the three GNEP reformulations.}
\label{tab:reform-roles}
\small
\begin{tabular}{@{}lll@{}}
\toprule
Reformulation & Role & Use in this paper\\
\midrule
\eqref{eq:G1} & Global-solution analysis &
Conceptual; not used directly for computation.\\
\eqref{eq:G2} & Stationarity, algorithm design &
Proximal alternating scheme (\S\ref{sec:algo-opt}).\\
\eqref{eq:MF} & Worst-case implementation &
NI-gap penalization (\S\ref{sec:algo-pess}).\\
\bottomrule
\end{tabular}
\end{table}

\section{Computational Workflow}\label{sec:algo}

This section specifies the concrete computational procedures used
in the case studies. We use \eqref{eq:G2} computationally for the
optimistic side (\S\ref{sec:algo-opt}) and a Nikaido--Isoda
penalization of \eqref{eq:MF} for the pessimistic side
(\S\ref{sec:algo-pess}). The frontier is constructed by a
structured scalarization sweep over weights.

\subsection{Optimistic candidate generation via proximal alternating
linearization}
\label{sec:algo-opt}
Let $(x^k,y^k,v^k)$ denote the current iterate with
$v^k\in S(x^k)$. Given $\tau_k>0$, the updated leader--follower
pair solves the proximal subproblem
\begin{equation}\label{eq:prox-sub}
  (x^{k+1},y^{k+1})
  \in\argmin_{(x,y)\in\mathcal K(x^k,v^k)}
  \Bigl\{F(x,y)+\tfrac{\tau_k}{2}\|x-x^k\|^2
  +\tfrac{\tau_k}{2}\|y-y^k\|^2\Bigr\},
\end{equation}
where
\begin{equation}\label{eq:Kk}
  \mathcal K(x^k,v^k):=\{(x,y)\in X\times Y:
  f(x,y)\le f(x^k,v^k)+\nabla_1 f(x^k,v^k)^\top(x-x^k)+\eps\}.
\end{equation}
The follower update is
$v^{k+1}\in\argmin_{v\in Y}f(x^{k+1},v)$. 

The stopping test is
$\max\{\|z^{k+1}-z^k\|,\ r_\mathrm{LL}^\eps(x^k,y^k),\
g_\mathrm{stat}^\eps(x^k,y^k,v^k,\lambda^k)\}\le\texttt{tol}$,
where $z^k=(x^k,y^k)$ and the stationarity residual uses the
Fischer--Burmeister form of the GNEP KKT
system~\eqref{eq:stat-residual}. Under the assumptions required
by the convexified GNEP reformulation of Lampariello and
Sagratella~\cite{lampariello2020numerically}---which include
fully convex lower-level objective, appropriate constraint
qualifications, and algorithm-specific conditions that we do not
verify in full for the case studies---this type of
stationarity-oriented scheme targets stationary points of the
optimistic problem. We use it here as a computational heuristic
rather than as a globally convergent algorithm, and we label
outputs accordingly.

\paragraph{Stationarity residual, formal definition.}
The stationarity residual $g_\mathrm{stat}^\eps(x,y,v,\lambda)$
used in the stopping test is defined through the GNEP KKT system
\eqref{eq:G1-kkt-x}--\eqref{eq:G1-kkt-comp}. Using the
Fischer--Burmeister function
$\varphi_{\mathrm{FB}}(a,b):=a+b-\sqrt{a^2+b^2}$, which satisfies
$\varphi_{\mathrm{FB}}(a,b)=0\iff a\ge 0,\ b\ge 0,\ ab=0$
\cite{fischer1992special,facchinei2003finite}, and the natural-residual
projections $R_X(z,w):=z-\proj_X(z-w)$ and $R_Y$ analogously, we
define
\begin{equation}\label{eq:stat-residual}
\begin{aligned}
  g_\mathrm{stat}^\eps(x,y,v,\lambda)^2
  :=\;& \|R_X(x,\nabla_x F(x,y)
        +\lambda(\nabla_x f(x,y)-\nabla_x f(x,v)))\|^2\\
  & + \|R_Y(y,\nabla_y F(x,y)+\lambda\nabla_y f(x,y))\|^2\\
  & + \|R_Y(v,\nabla_v f(x,v))\|^2\\
  & + \varphi_{\mathrm{FB}}(\lambda,\,-[f(x,y)-f(x,v)-\eps])^2.
\end{aligned}
\end{equation}
This quantity uses only first derivatives of $F$ and $f$ and
requires no derivative of $\phi$. It is evaluated at the
multiplier $\lambda\ge 0$ returned by the proximal subproblem
solver without the optimization over $\lambda$.

\subsection{Pessimistic evaluation via NI-gap penalization}
\label{sec:algo-pess}
For a candidate decision $x\in X$, we evaluate $\psi_\eps^p(x)$ by
applying a penalty to the Nikaido--Isoda gap~\eqref{eq:NI}. This
is the concrete computational role of the lower-equilibrium
reformulation.

\paragraph{Inner evaluation.}
Fix $x$. For a penalty $\sigma>0$, compute
\begin{equation}\label{eq:ni-penalty}
  (y^\sigma,v^\sigma)
  \in\argmin_{y\in Y,v\in Y}
  \Bigl\{-F(x,y)+\sigma\cdot\bigl[\mathcal N_x(y,v)\bigr]_+\Bigr\},
\end{equation}
where $[\cdot]_+$ denotes nonnegative truncation (included for
numerical safety, since $\mathcal N_x(y,v)\ge 0$ throughout).
In the compact setting, exact global minimization of $-F(x,y) + \sigma \mathcal{N}_x(y,v)$ along a schedule $\sigma_t\uparrow \infty$ yields accumulation points with vanishing NI gap under standard penalty arguments~\cite{fukushima1992equivalent,harker1991generalized,von2009optimization}, so that $F(x,y^\sigma)\to\psi_\eps^p(x)$.
In our numerical experiments, however, the penalized
subproblems are nonconvex and are solved locally with SLSQP from
multiple starts; the reported pessimistic values should
therefore be interpreted as certified only by the achieved NI
gaps and status labels, not as global pessimistic values. We use
$\sigma\in\{10^1,10^2,10^3,10^4\}$ and stop when $\mathcal N_x$
falls below $10^{-6}$; the achieved gap is reported alongside
each pessimistic evaluation in the supplementary code.

\paragraph{Outer search.}
To solve $\min_{x\in X}\psi_\eps^p(x)$, we combine the NI-gap
inner evaluation with a derivative-free outer search using
multistart Powell / Nelder--Mead, as in
\cite{conn2009introduction,audet2017derivative}. We report multistart
counts, starting points, and status labels with the results.

\subsection{Frontier construction via scalarization sweep}
\label{sec:algo-frontier}
For each weight $\omega\in[0,1]$ define the scalarized leader
problem
\begin{equation}\label{eq:scalar-sweep}
  \min_{x\in X}\Bigl\{(1-\omega)\psi_\eps^o(x)+\omega\,\Delta_\eps(x)\Bigr\}.
\end{equation}
Let $\omega$ range over a deterministic grid
$\omega_j=j/(J-1)$, $j=0,\dots,J-1$. 
Each subproblem \eqref{eq:scalar-sweep} is solved by a black-box multistart
procedure, with inner evaluations of $\psi_\eps^o$ and
$\psi_\eps^p$ as above. The weighted-sum sweep recovers
\emph{supported} nondominated points, i.e.\ points on the convex
hull of the attainable objective set~\cite{ehrgott2005multicriteria,
miettinen1999nonlinear}. In nonconvex objective images, it need not
recover unsupported portions of the frontier, which motivates
the Latin-hypercube fill of $N_\mathrm{LHS}$ additional samples.
Compared to unstructured uniform sampling, this design covers
the supported frontier explicitly through the $\omega$-grid and
fills the interior more efficiently~\cite{mckay2000comparison,
viana2016tutorial}.

\paragraph{Status labels.} Consistent with standard computational reporting
conventions~\cite{dolan2002benchmarking,gould2015cutest}, we label
outputs as \emph{converged}, \emph{incumbent}, \emph{heuristic},
or \emph{empirical Pareto}. Only \emph{converged} outputs have
passed the full stopping test; \emph{incumbents} are returned at
iteration limits; \emph{heuristic} denotes manually specified
benchmarks; \emph{empirical Pareto} denotes undominated points
within the sampled candidate set.

\subsection{Assumptions used where}
\label{sec:assum-map}
Because the paper moves between compactness, Lipschitz continuity,
quadratic growth, differentiability, convexity, and GNEP
stationarity, Table~\ref{tab:assum-map} maps each result to the
assumptions it requires. The standing assumptions
(Assumption~\ref{ass:standing}: $X,Y$ compact and $F,f$
continuous) are used throughout; additional assumptions are
introduced only where they are needed.

\begin{table}[ht]
\centering
\caption{Map of assumptions to results.}
\label{tab:assum-map}
\small
\begin{tabular}{@{}p{3.6cm}p{6.0cm}p{4.6cm}@{}}
\toprule
Result & Main assumptions & Purpose \\
\midrule
Prop.~\ref{prop:attain} & Standing (compactness, continuity) & Attainment of $\psi^o,\psi^p$; one-dim.\ diameter form. \\
Prop.~\ref{prop:monotone} & Standing & Monotonicity in $\eps$. \\
Thm.~\ref{thm:diam-bound} & Standing; $F(x,\cdot)$ Lipschitz on $Y$ with modulus $L_F(x)$ & Diameter-based bound. \\
Cor.~\ref{cor:sqrt-rate} & Above plus quadratic lower-level growth (Ass.~\ref{ass:qg}) & $O(\sqrt\eps)$ rate. \\
Eq.~\eqref{eq:directional}--\eqref{eq:directional-sym} & Standing; $f$ twice continuously differentiable; second-order sufficient condition at $y^\star(x)$; $S(x)$ singleton & Directional expansion. \\
Prop.~\ref{prop:continuity}, semicontinuity part & Standing & USC of $S_\eps$, USC of $\psi^p$, LSC of $\psi^o$. \\
Prop.~\ref{prop:continuity}, continuity part & Above plus correspondence continuity of $S_\eps$ at $(x_0,\eps_0)$ & Continuity of $\psi^o,\psi^p,\Delta_\eps$. \\
GNEP KKT systems (\ref{eq:G1-kkt-x})--(\ref{eq:G1-kkt-comp}), (\ref{eq:MF-kkt-y})--(\ref{eq:MF-kkt-comp}) & $F,f$ continuously differentiable; appropriate constraint qualifications & First-order conditions. \\
Stationarity residual~\eqref{eq:stat-residual} & Continuous differentiability of $F,f$ on a neighborhood of the candidate point & Computational diagnostic. \\
Proximal scheme~\eqref{eq:prox-sub} & Convexity assumptions of~\cite{lampariello2020numerically} (full lower-level convexity, CQ, etc.) & Stationarity-targeting heuristic. \\
NI penalization~\eqref{eq:ni-penalty} & Compactness; exact global minimization (theory) / local SLSQP (practice) & Pessimistic evaluation. \\
\bottomrule
\end{tabular}
\end{table}

\section{Case Study 1: Stackelberg Pricing on a Parallel-Link
Network with Genuine Follower Multiplicity}
\label{sec:case1}
This case study is designed to exercise the diagnostic on an
instance where the follower's optimal response is genuinely
set-valued along a positive-dimensional, codimension-one
indifference subset of the leader's decision region. It is
deliberately small, so that all quantities can be written down
in closed form and the framework can be compared against exact
values.

\subsection{Model}
A planner sets tolls $x=(x_1,x_2)\in X:=[0,2]^2$ on two parallel
links between a single origin-destination pair with unit demand.
Link $i$ has an access cost $a_i$; we use $a=(1.0,1.2)$. A
representative user chooses flow split $y=(y_1,y_2)\in Y$ with
\begin{equation}\label{eq:case1-Y}
  Y:=\{y\in\R_+^2:y_1+y_2=1\},
\end{equation}
and minimizes the linear travel cost
\begin{equation}\label{eq:case1-f}
  f(x,y):=(a_1+x_1)y_1+(a_2+x_2)y_2.
\end{equation}
When $a_1+x_1\ne a_2+x_2$, the follower's optimum is the unique
unit mass on the cheaper link; when $a_1+x_1=a_2+x_2$, every
$y\in Y$ is optimal. Thus the follower's exact optimal set is
\begin{equation}\label{eq:case1-S}
  S(x)=
  \begin{cases}
    \{(1,0)\} & \text{if }a_1+x_1<a_2+x_2,\\
    \{(0,1)\} & \text{if }a_1+x_1>a_2+x_2,\\
    Y         & \text{if }a_1+x_1=a_2+x_2.
  \end{cases}
\end{equation}
Along the indifference line $\mathcal L:=\{x\in X:x_2=x_1-0.2\}$,
the follower's exact optimum is an entire line segment.

The planner has a system-cost objective that balances congestion
and revenue:
\begin{equation}\label{eq:case1-F}
  F(x,y):=
  c_1 y_1^2+c_2 y_2^2
  -\alpha\,(x_1 y_1+x_2 y_2)
  +\beta\,(x_1^2+x_2^2).
\end{equation}
We use $c_1=1.5$, $c_2=1.0$, $\alpha=0.3$, $\beta=0.05$. The term
$c_1 y_1^2+c_2 y_2^2$ is a stylized quadratic congestion cost,
$\alpha(x\cdot y)$ is toll revenue valued by the planner, and
$\beta\|x\|^2$ discourages extreme tolls. Both $X$ and $Y$ are
compact; $F$ is quadratic and $f$ is linear, so
Assumptions~\ref{ass:standing} and the Lipschitz condition on
$F(x,\cdot)$ are satisfied.

\subsection{Explicit ambiguity premium}
Along $\mathcal L$ the set $S(x)=Y$ is the whole simplex, so
\begin{equation}\label{eq:case1-delta-L}
  \Delta_0(x)=\max_{y\in Y}F(x,y)-\min_{y\in Y}F(x,y)
  \qquad\text{for }x\in\mathcal L.
\end{equation}
With unit demand, parameterize $y=(t,1-t)$ for $t\in[0,1]$. The
leader's objective on $Y$ for $x\in\mathcal L$ is
\begin{equation}\label{eq:case1-F-t}
  F(x,(t,1-t))
  =c_1 t^2+c_2(1-t)^2-\alpha[x_1 t+x_2(1-t)]+\beta\|x\|^2,
\end{equation}
a strictly convex quadratic in $t$. Expanding,
\begin{equation}\label{eq:case1-F-quadratic}
  F(x,(t,1-t))
  =(c_1+c_2)\,t^2
  +\bigl[-2c_2-\alpha(x_1-x_2)\bigr]\,t
  +\text{const}(x),
\end{equation}
where $\text{const}(x)=c_2-\alpha x_2+\beta\|x\|^2$. The first-order
condition $\partial_t F=0$ reads
$2(c_1+c_2)\,t-2c_2-\alpha(x_1-x_2)=0$, so the unconstrained
interior minimizer is
\begin{equation}\label{eq:case1-tstar}
  t^\star(x)
  =\frac{2c_2+\alpha(x_1-x_2)}{2(c_1+c_2)}
  =\frac{c_2+\tfrac12\alpha(x_1-x_2)}{c_1+c_2},
\end{equation}
clipped to $[0,1]$; the right-hand equality is the form used in
the supplementary code, obtained by dividing numerator and
denominator by $2$. The maximum of the strictly convex quadratic
on $[0,1]$ is attained at an endpoint $t\in\{0,1\}$. For the
nominal-optimal point $x=(1.6,1.4)\in\mathcal L$, which minimizes
$\psi_0^o$ on $\mathcal L\cap X$,
\eqref{eq:case1-tstar} gives
$t^\star=(2\cdot 1.0+0.3\cdot 0.2)/(2\cdot 2.5)=2.06/5.0=0.412$,
and the pessimistic maximum is attained at $t=0$, giving
\begin{equation}\label{eq:case1-exact-delta}
  \Delta_0((1.6,1.4))=
  F((1.6,1.4),(1,0))-F((1.6,1.4),(0.412,0.588))
  \approx 0.864,
\end{equation}
a large, strictly positive ambiguity premium at exact follower
optimality. Off $\mathcal L$, the exact $S(x)$ is a singleton and
$\Delta_0(x)=0$.

\subsection{$\eps$-neighborhood and the full framework}
For $\eps>0$, whether on or off $\mathcal L$, the set $S_\eps(x)$
is characterized explicitly. If $a_1+x_1<a_2+x_2$, writing
$\delta(x):=(a_2+x_2)-(a_1+x_1)>0$, one has
\begin{equation}\label{eq:case1-Seps}
  S_\eps(x)=\{y=(1-s,s):\ s\in[0,\min(1,\eps/\delta(x))]\}.
\end{equation}
This allows one to compute $\psi_\eps^o$, $\psi_\eps^p$, and
$\Delta_\eps$ in closed form along any ray in $X$. We use
$\eps=0.10$ in the main reported run.

\subsection{Scalarization sweep and frontier}
We solve~\eqref{eq:scalar-sweep} for $J=21$ weights
$\omega_j=j/20$, $j=0,\dots,20$, augmented with
$N_\mathrm{LHS}=80$ Latin-hypercube samples on $X$.
Table~\ref{tab:case1-diag} reports diagnostics at seven
representative points: five scalarization snapshots
($\omega\in\{0,0.25,0.5,0.75,1\}$), and two heuristic
benchmarks. Values are computed in closed form
from~\eqref{eq:case1-F-t}--\eqref{eq:case1-Seps} and verified by
the supplementary code.

\begin{table}[ht]
\centering
\caption{Case Study 1 diagnostics at representative leader
decisions, $\eps=0.10$. All values in model units. ``$\omega$''
denotes the scalarization weight. ``On $\mathcal L$'' indicates
whether $x$ satisfies $x_2=x_1-0.2$.}
\label{tab:case1-diag}
\small
\begin{tabular}{@{}lcccccc@{}}
\toprule
Label & $x_1$ & $x_2$ & On $\mathcal L$ &
$\psi_\eps^o(x)$ & $\Delta_\eps(x)$ & $\rho_\eps(x)$\\
\midrule
$\omega=0$ (nominal, on $\mathcal L$) & 1.60 & 1.40 & yes & 0.382 & 0.864 & 0.626\\
$\omega=0.25$                         & 1.60 & 1.20 & no  & 0.391 & 0.449 & 0.323\\
$\omega=0.5$                          & 1.60 & 1.00 & no  & 0.489 & 0.389 & 0.261\\
$\omega=0.75$                         & 1.60 & 0.50 & no  & 0.762 & 0.228 & 0.129\\
$\omega=1$ (robust corner)            & 2.00 & 0.00 & no  & 1.063 & 0.137 & 0.066\\
Heur. no toll                         & 0.00 & 0.00 & no  & 0.625 & 0.875 & 0.538\\
Heur. symmetric                       & 1.00 & 1.00 & no  & 0.425 & 0.875 & 0.614\\
\bottomrule
\end{tabular}
\end{table}

The key qualitative phenomena are visible in
Table~\ref{tab:case1-diag}. First, the nominal-optimal policy sits
on the indifference line $\mathcal L$ and has the smallest
$\psi_\eps^o$ of the set ($0.382$) but a large ambiguity premium
($0.864$) driven by genuine follower multiplicity. Second, moving
the leader's second toll $x_2$ away from the indifference value
$x_1-0.2$ trades nominal efficiency for robustness in a
continuous way: as $x_2$ decreases from $1.40$ to $0.00$ along
the $\omega$-sweep, $\psi_\eps^o$ rises from $0.382$ to $1.063$
while $\Delta_\eps$ falls from $0.864$ to $0.137$. Third, at both
heuristic benchmarks the cost gap is $\delta=0.2$, so with
$\eps=0.10$ the admissible parameter range is
$s\in[0,\eps/\delta]=[0,0.5]$, i.e.\ a \emph{half}-simplex of
flow splits. Directly computing $\psi_\eps^o$ and $\psi_\eps^p$
on this half-simplex yields $\Delta_\eps=0.875$ at both $(0,0)$
and $(1,1)$---a substantial ambiguity that is only partially
absorbed by the codimension-one geometry of $\mathcal L$. This
is precisely the phenomenon the diagnostic is designed to
surface.

Figure~\ref{fig:case1-frontier} displays the computed
robustness--efficiency frontier in the
$(\psi_\eps^o,\Delta_\eps)$ plane. The scalarization sweep
populates a well-resolved curve, with the nominal corner on
$\mathcal L$, the robust corner off $\mathcal L$, and a dense
collection of compromise points in between. The supported points
form a smooth decreasing trade-off curve, reflecting the
classical efficiency--robustness tension.

\begin{figure}[ht]
  \centering
  \includegraphics[width=0.65\textwidth]{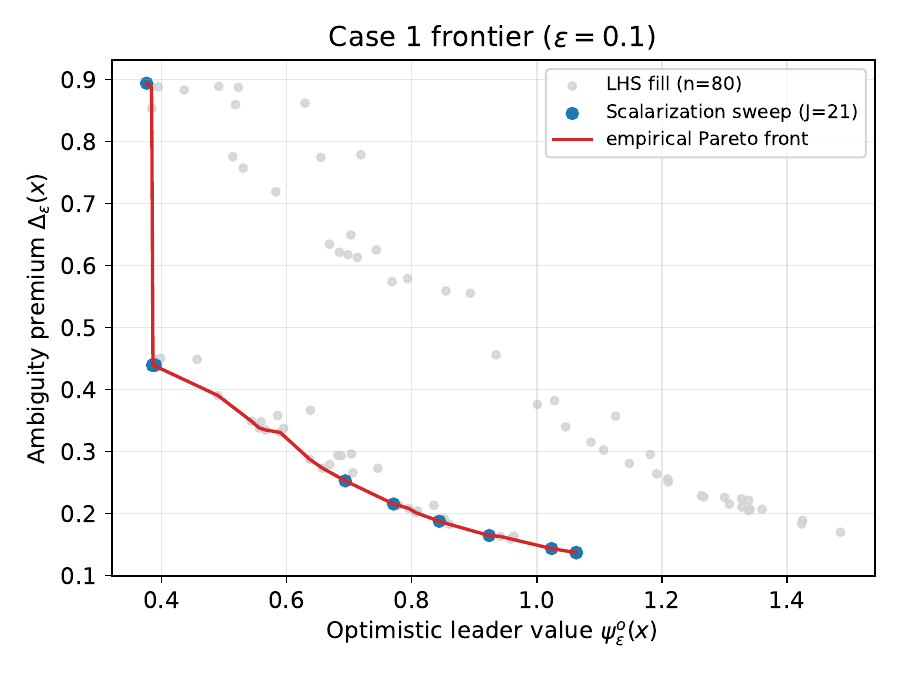}
  \caption{Robustness--efficiency frontier for Case Study 1
  ($\eps=0.10$). Blue dots: scalarization sweep over 21 weights.
  Gray dots: 80 Latin-hypercube samples. The nominal-optimal
  corner lies on the indifference line $\mathcal L$; the
  robustness-optimal corner lies strictly off $\mathcal L$.}
  \label{fig:case1-frontier}
\end{figure}

\subsection{Interpretation}
Case Study 1 demonstrates the framework in the mode it was
motivated to serve: a Stackelberg game with a follower whose
best response is genuinely set-valued along a
positive-dimensional, codimension-one indifference subset of
the leader's decision region. Here the ambiguity premium
$\Delta_\eps(x)$ is not a stability artifact but a real feature
of the follower's optimal correspondence, and the
robustness--efficiency frontier displays both regimes---a
high-reward brittle corner and a safer interior---with
intermediate compromise policies populating the curve.

\section{Case Study 2: Convex Generation-Capacity Planning with
Reliability-Type Diversification Constraints}
\label{sec:case2}
Building on the stylized convex generation-planning model, this second case study introduces reliability-motivated diversification constraints to rule out unrealistic single-technology portfolios.
Methodologically, we use the structured scalarization sweep of
\S\ref{sec:algo-frontier} rather than unstructured uniform random
sampling. 
Furthermore, rather than focusing on prescriptive managerial claims, 
this section is specifically designed to illustrate how the framework behaves
when multiplicity is absent and ambiguity arises only from the
$\eps$-relaxation, consistent with Corollary~\ref{cor:sqrt-rate}.

\subsection{Model}
We use a four-technology planning model with utility-scale solar
photovoltaics, onshore wind, natural-gas combined cycle, and a
stylized four-hour battery-like flexibility resource. The
technology data follow the open NREL ATB conventions and are
reported in Table~\ref{tab:techdata} for transparency;
quantitative values should not be interpreted as empirical
forecasts.

\begin{table}[ht]
\centering
\caption{Stylized technology data. Annualized capital assumes 7\%
discount rate, 25-year life.}
\label{tab:techdata}
\small
\begin{tabular}{@{}lrrrr@{}}
\toprule
Technology & CAPEX (\$/kW) & Ann. cap. (\$B/GW-yr) &
Var. cost (\$/MWh) & Cap. factor\\
\midrule
Solar PV     & 1{,}194 & 0.1025 &  5.0 & 0.246\\
Onshore Wind & 1{,}462 & 0.1255 &  7.0 & 0.345\\
Gas CCGT     &   958 & 0.0822 & 28.6 & 0.870\\
Battery 4h   & 1{,}247 & 0.1070 &  8.2 & 0.200\\
\bottomrule
\end{tabular}
\end{table}

\paragraph{Leader and follower.}
Leader variables $x=(x_1,\ldots,x_4)$ are installed capacities in
GW; follower variables $y=(y_1,\ldots,y_4)$ are average dispatches
in GW. Following the standing framework, the leader's effective
feasible set is a fixed compact subset of $\R^4$, and the
follower's feasible set $Y$ is a fixed compact box independent of
$x$:
\begin{equation}\label{eq:case2-XY}
\begin{aligned}
  X_{\mathrm{box}} &:=[0.2,8.0]\times[0.2,6.0]\times[0.5,10.0]\times[0.1,4.0],\\
  Y &:=[0,6.0]\times[0,5.0]\times[0,9.0]\times[0,3.5].
\end{aligned}
\end{equation}
The dispatch box $Y$ is sized generously so that the
unconstrained per-technology optimal dispatch $y_i=\alpha_i x_i$
is interior to $Y$ for every $x\in X_{\mathrm{box}}$ used in the
case study. The coupling between capacities and dispatches enters
\emph{soft}, through the lower-level objective's penalization of
$|y_i-\alpha_i x_i|$, rather than through a feasibility
restriction $y_i\le x_i$; this keeps $Y$ leader-independent and
preserves the standing assumptions of Section~\ref{sec:problem}.

Lower-level objective
\begin{equation}\label{eq:case2-f}
  f(x,y)=\sum_i w_i(y_i-\alpha_i x_i)^2
  +\mu_D\Bigl(\sum_i y_i-D\Bigr)^2,
\end{equation}
with $\alpha=(0.246,0.345,0.870,0.200)$,
$w=(3.0,2.5,1.0,2.0)$, $\mu_D=2.0$, and $D=5.0$ GW.

Leader objective
\begin{equation}\label{eq:case2-F}
  F(x,y)=\sum_i\kappa_i x_i+\sum_i\beta_i(x_i-y_i)^2
  +\sum_i\tilde c_i y_i+\lambda_D\Bigl(\sum_i y_i-D\Bigr)^2,
\end{equation}
with $\kappa$ from Table~\ref{tab:techdata}, $\beta=(0.15,0.12,0.05,0.08)$,
$\tilde c_i$ annualized dispatch-cost coefficients, and
$\lambda_D=1.5$.

\paragraph{Reliability-type diversification constraints.}
We add the constraints
\begin{equation}\label{eq:case2-rely}
  \sum_i x_i\ge D,
  \qquad
  x_i\le 0.6\sum_j x_j \quad \forall i,
  \qquad
  x_i\ge \underline x_i\ (\text{min build}),
\end{equation}
with $\underline x=(0.2,0.2,0.5,0.1)$ GW. The first constraint
enforces nameplate adequacy; the second caps any single
technology's share below $60\%$, a stylized diversification
criterion consistent with planning practice in many
jurisdictions~\cite{kahrobaee2010short,joskow2011comparing,billimoria2019market,
macdonald2016future,denholm2021challenges}; the third prevents zero
build of any technology. These constraints preclude the trivial
all-gas corner reported in the earlier version and make the
robustness--efficiency trade-off nontrivial. The effective leader
feasible set is therefore
\begin{equation}\label{eq:case2-Xfeas}
  X_{\mathrm{feas}}
  :=\Bigl\{x\in X_{\mathrm{box}}:\
  \sum_i x_i\ge D,\ \,
  x_i\le 0.6\sum_j x_j\ \forall i\Bigr\},
\end{equation}
which is a closed compact subset of $\R^4$. In the notation of
Sections~\ref{sec:problem}--\ref{sec:algo}, we take
$X=X_{\mathrm{feas}}$ throughout this case study.

\paragraph{Follower uniqueness.}
Because~\eqref{eq:case2-f} is strictly convex in $y$, $S(x)$ is a
singleton for every $x$. By Corollary~\ref{cor:sqrt-rate}, the
ambiguity premium satisfies
$\Delta_\eps(x)\le 2L_F(x)\sqrt{\eps/\mu(x)}$. A valid
quadratic-growth constant under the convention of
Assumption~\ref{ass:qg} (the inequality
$f(x,y)-\phi(x)\ge\mu(x)\,\dist(y,S(x))^2$, not Hessian
strong-convexity) is $\mu(x)=\min_i w_i=1.0$. The case study
therefore illustrates the framework in the regime where
multiplicity is absent and ambiguity arises from $\eps>0$,
exactly as the theory predicts.

\subsection{Results}
We use the workflow of Section~\ref{sec:algo} to compute a
candidate nominal optimum (which solves
$\min_{x\in X_\mathrm{feas}}\psi_\eps^o(x)$ by multistart Powell
with NI-penalized inner pessimistic evaluation), a candidate
robust optimum (which solves
$\min_{x\in X_\mathrm{feas}}\Delta_\eps(x)$ similarly), and a
family of convex-combination intermediates along the segment
between them. Latin-hypercube samples ($N_\mathrm{LHS}=80$)
populate the interior of the feasible region. The sweep-plus-fill
design produces the frontier in Figure~\ref{fig:case2-frontier}.
Table~\ref{tab:case2-policies} reports diagnostics at seven
representative policies. We explicitly note that $\psi_\eps^o(x)$
is a function of $x$, so entries in the $\psi^o$ column are
evaluations of the same leader functional at different
portfolios, not competing values of a single optimization.

\begin{table}[ht]
\centering
\caption{Case Study 2 diagnostics at representative policies,
$\eps=0.5$. Diversification constraints active (share cap 0.6,
reserve margin $\sum x_i\ge D=5$, minimum builds). Capacities
$x$ in GW; cost units \$B/yr. The ``status'' column labels the
computational provenance: \emph{incumbent} denotes the best
multistart solution for the indicated objective, and
\emph{convex comb.}\ denotes a policy on the line segment
between the two incumbents.}
\label{tab:case2-policies}
\small
\begin{tabular}{@{}lcccccccc@{}}
\toprule
Label & $x_1$ & $x_2$ & $x_3$ & $x_4$ &
$\psi^o$ & $\Delta$ & $\rho$ & Status\\
\midrule
Solver incumbent for nominal solve  & 1.09 & 0.35 & 2.58 & 0.98 & 0.599 & 1.129 & 0.706 & incumbent \\
Convex comb.\ $t{=}0.25$             & 1.08 & 0.54 & 3.18 & 1.03 & 0.681 & 0.900 & 0.535 & convex comb.\ \\
Convex comb.\ $t{=}0.50$             & 1.08 & 0.72 & 3.79 & 1.08 & 0.799 & 0.695 & 0.387 & convex comb.\ \\
Convex comb.\ $t{=}0.75$             & 1.07 & 0.91 & 4.39 & 1.12 & 0.949 & 0.518 & 0.266 & convex comb.\ \\
Solver incumbent for robust solve   & 1.06 & 1.09 & 4.99 & 1.17 & 1.132 & 0.370 & 0.174 & incumbent \\
Heur.\ gas-heavy (cap-binding)      & 0.79 & 0.79 & 3.12 & 0.50 & 0.596 & 0.922 & 0.578 & heuristic \\
Heur.\ renewables-heavy             & 3.00 & 2.50 & 1.00 & 0.50 & 1.451 & 1.460 & 0.596 & heuristic \\
\bottomrule
\end{tabular}
\end{table}

Table~\ref{tab:case2-policies} illustrates several phenomena
worth highlighting. First, among the scalarization-generated
sweep points, the solver incumbent for the nominal solve has the
second smallest $\psi^o$ and the second largest $\Delta$, whereas the solver
incumbent for the robust solve has the second largest $\psi^o$ and the
smallest $\Delta$; the convex-combination intermediates trace
out a monotone trade-off.

Second, Table~\ref{tab:case2-policies} also illustrates the
diagnostic value of reporting sampled heuristics alongside solver
incumbents. The gas-heavy heuristic slightly improves on the
reported solver incumbent for the nominal solve in both
displayed criteria, with
$(\psi^o,\Delta)=(0.596,0.922)$ compared with
$(0.599,1.129)$. 
Thus, within the sampled set, the solver
incumbent is not empirically nondominated. We keep both rows in
the table to make the computational status transparent: the
``solver incumbent'' is the best multistart output for the
nominal solve, whereas the heuristic is a hand-specified
benchmark that exposes a missed or weak local solution. This
reinforces the paper's diagnostic stance and motivates reporting
status labels rather than presenting the multistart outputs as
certified global optima.

Third, the $60\%$ share cap binds at the solver incumbent for
the robust solve ($x_3/\sum_j x_j=4.99/8.31\approx 0.600$) and
at the gas-heavy heuristic ($3.12/5.20=0.600$), but not at the
solver incumbent for the nominal solve (gas share
$2.58/5.00\approx 0.516$) nor at the renewables-heavy heuristic
(largest share $\approx 0.429$ in solar). No portfolio
degenerates to a single-technology corner, and the constraint is
tight where one would expect: in the gas-dominated regime
corresponding to high robustness in this stylized model.

The $O(\sqrt\eps)$ prediction of Corollary~\ref{cor:sqrt-rate} is
examined empirically in Figure~\ref{fig:case2-sqrt}, which plots
the ratio $\Delta_\eps(x)/\sqrt\eps$ at the balanced policy
$x=(1.08,0.72,3.79,1.08)$ for
$\eps\in\{0.05,0.1,0.25,0.5,1,2,4\}$. The computed ratios are
$0.810,\ 0.830,\ 0.892,\ 0.983,\ 1.129,\ 1.351,\ 1.688$. The
ratio is relatively stable for small $\eps$ and rises for larger
tolerances, consistent with a leading-order $\sqrt\eps$ effect
plus higher-order corrections that become non-negligible as
$\eps$ approaches unity in problem units.

\begin{figure}[ht]
  \centering
  \includegraphics[width=0.65\textwidth]{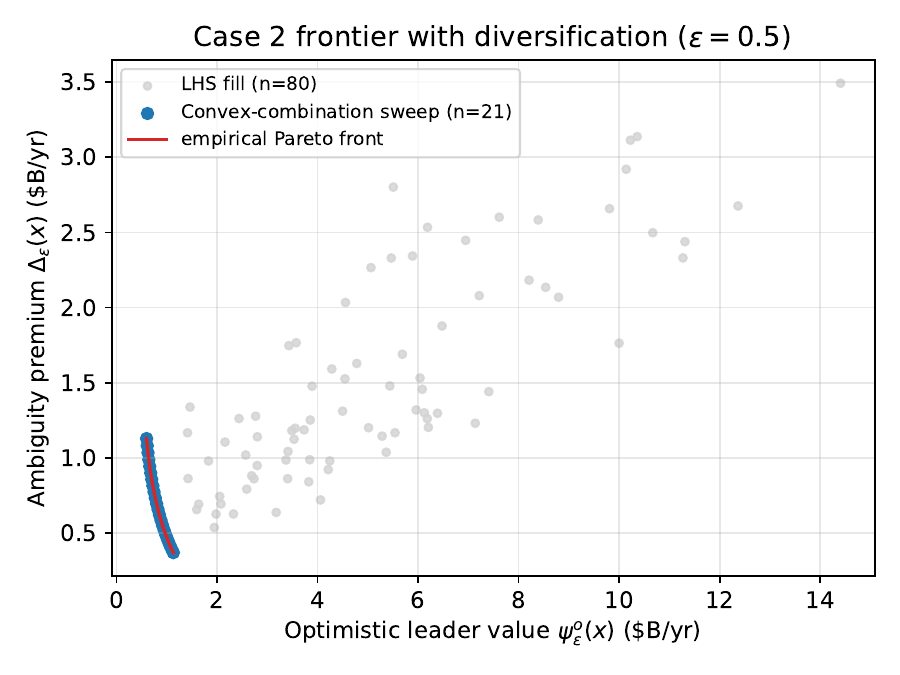}
  \caption{Robustness--efficiency frontier for Case Study 2
  ($\eps=0.5$), with diversification constraints. Blue dots:
  scalarization sweep. Gray dots: Latin-hypercube fill.}
  \label{fig:case2-frontier}
\end{figure}

\begin{figure}[ht]
  \centering
  \includegraphics[width=0.65\textwidth]{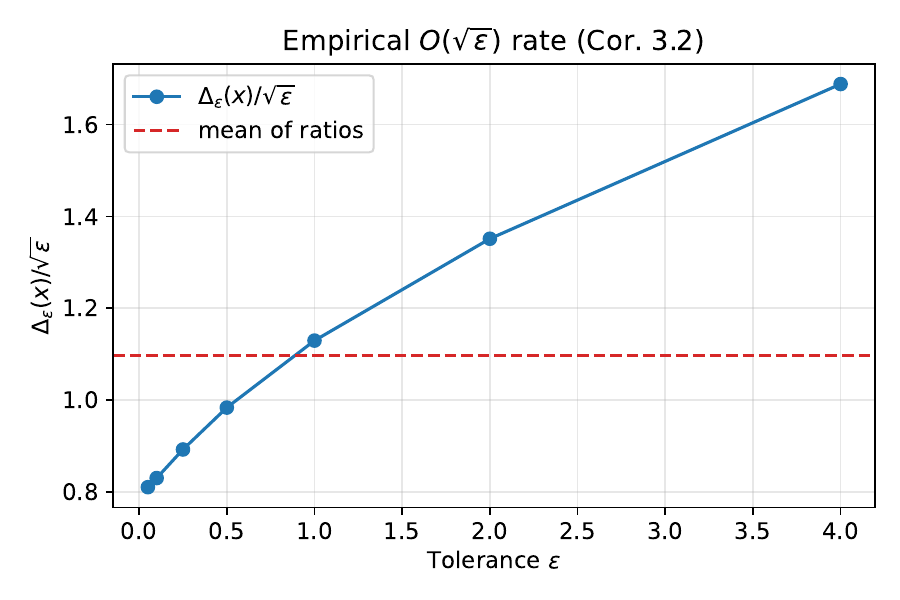}
  \caption{Empirical check of the $O(\sqrt\eps)$ rate
  (Corollary~\ref{cor:sqrt-rate}): ratio $\Delta_\eps/\sqrt\eps$
  for the balanced policy across a grid of $\eps$. The ratio is
  relatively stable for small $\eps$ and rises for larger
  tolerances, consistent with a leading-order $\sqrt\eps$ effect
  plus higher-order corrections.}
  \label{fig:case2-sqrt}
\end{figure}

\subsection{Interpretation}
The generation-planning example illustrates a different regime
from Case Study 1. The lower-level problem is strictly convex, so
exact follower multiplicity is absent. Any ambiguity premium
therefore comes from the $\eps$-optimal response set rather than
from multiple exact optima. The purpose of the example is to
show how the diagnostic behaves when implementation slack, rather
than exact multiplicity, is the source of exposure.

With diversification constraints in place, no portfolio
degenerates to a single-technology corner. The multistart
workflow identifies a solver incumbent for the nominal solve
with a gas-heavy but not share-cap-binding mix---roughly $2.58$
GW CCGT out of a $5.00$ GW total, with gas accounting for about
$52\%$ of total capacity---and a solver incumbent for the robust
solve that increases total capacity substantially (to about
$8.31$ GW) and pushes the gas share to the diversification cap
($\approx 60\%$), while also increasing wind and battery
capacity in absolute terms relative to the nominal solver
incumbent. We do not claim that either is the right investment
recommendation for any real system. What we claim is only that:
\begin{enumerate}[leftmargin=*,nosep]
\item even without follower multiplicity, the ambiguity premium
  is strictly positive because the $\eps$-response set has
  positive diameter (Corollary~\ref{cor:sqrt-rate});
\item the shape of the frontier is informative for screening,
  identifying policies that are nominally attractive but
  proportionally more exposed to near-optimal redispatch;
\item the framework is sensitive to the choice of
  diversification or reliability constraints. Without such
  constraints, the earlier all-gas pessimistic solution we
  documented in a preliminary version of this paper was an
  artifact of the model's omission of reliability and
  contingency dimensions, not a substantive recommendation.
\end{enumerate}
This illustrates a general point: the ambiguity premium is one
dimension of robustness among several (contingency, outage,
uncertainty). A well-posed planning problem must combine it with
the dimensions captured by standard reliability and
uncertainty-aware planning
models~\cite{conejo2010decision,mulvey1995robust,delage2010distributionally}.

\section{Managerial Insights}
\label{sec:managerial}
The diagnostic framework supports decision-makers in three
concrete ways. First, by reporting each candidate decision as a
triple $(\psi_\eps^o(x),\Delta_\eps(x),g_\mathrm{stat}^\eps)$, it
separates nominal efficiency, implementation robustness, and
first-order computational quality that are often conflated in
single-solution bilevel reports. Second, by visualizing the
robustness--efficiency frontier generated via a structured
scalarization sweep, it exposes whether a candidate decision
sits on a brittle corner or a robust plateau. Third, by
stress-testing $\Delta_\eps(x)$ across a grid of $\eps$, it
quantifies the sensitivity of a decision to implementation
slack, forecast error, or near-optimal follower re-response.

As Case Study 2 illustrates, in
models that omit reliability, outage, or network constraints,
pessimistic optima can collapse to single-technology corners
that minimize redispatch ambiguity only because the redispatch
space itself has collapsed. The appropriate use of the framework
is as a screening overlay on top of a planning model that is
already well-posed with respect to the other dimensions of
robustness a decision-maker cares about.

Within that scope, the framework applies naturally to network
pricing and transportation~\cite{labbe1998bilevel,brotcorne2001bilevel,
gilbert2015numerical}, electricity market
design~\cite{daxhelet2007eu,hobbs2007nash,gabriel2012complementarity,wogrin2020applications}, Stackelberg
investment and entry games~\cite{moiseeva2014exercise}, and regulatory
decision-making in adversarial or strategic
environments~\cite{smith2020survey,caprara2016bilevel,lozano2017value}, as long as the analyst is willing to
treat it as one component of a larger decision-analytic
process.

\section{Discussion and Conclusion}
\label{sec:conclusion}
We have proposed a diagnostic framework for hierarchical
decision problems that organizes known optimistic and pessimistic
bilevel--GNEP reformulations around the ambiguity premium
$\Delta_\eps(x)$. The novelty is not the gap itself, which is
classical, but the decision-analytic packaging around it: the
Lipschitz-type bound relating $\Delta_\eps(x)$ to the diameter of
the $\eps$-response set, the $O(\sqrt\eps)$ rate under quadratic
lower-level growth, the structured scalarization sweep for
frontier construction, and the combined use of the convexified
optimistic reformulation and Nikaido--Isoda pessimistic
evaluation as a screening workflow.

Several limitations should be noted. The stationarity residual
\eqref{eq:stat-residual} uses Fischer--Burmeister encoding, which
is standard but in the presence of degenerate multipliers can
fail to separate stationarity from numerical noise; more refined
encodings are
available~\cite{kanzow1996nonlinear,kanzow1998new,ralph1995directional}. The
pessimistic outer solve is derivative-free multistart and does
not return global certificates. The case studies are small and
are chosen to expose the framework's mechanics; scaling the
Nikaido--Isoda penalization to high-dimensional problems is an
area of active research~\cite{dreves2016solving}.

Natural extensions include stochastic and robust lower-level
models~\cite{beck2023survey,buchheim2022stochastic,
burtscheidt2020risk}, network-constrained dispatch with
unit-commitment multiplicity~\cite{o2005efficient,bertsimas2012adaptive,pozo2011finding}, and coupling with
risk-averse leader objectives~\cite{rockafellar2000optimization,
artzner1999coherent}. In all these directions the ambiguity premium can
play the role of a compact diagnostic that exposes
implementation risk that is otherwise easy to overlook.

\appendix
\section{Proofs}\label{sec:proofs}

\begin{proof}[Proof of Proposition~\ref{prop:continuity}]
\textbf{Upper semicontinuity of $S_\eps$.}
Let $(x_k,\eps_k)\to(x,\eps)$ and let $y_k\in S_{\eps_k}(x_k)$
with $y_k\to y$. Since $Y$ is compact, $y\in Y$. By continuity of
$f$ and $\phi$,
\[
f(x,y)=\lim_{k\to\infty}f(x_k,y_k)
\le \lim_{k\to\infty}\bigl(\phi(x_k)+\eps_k\bigr)
= \phi(x)+\eps.
\]
Hence $y\in S_\eps(x)$. Therefore the graph of $S_\eps$ is
closed. Since $S_\eps(x)\subseteq Y$ and $Y$ is compact, each
value $S_\eps(x)$ is compact and nonempty. Thus
$(x,\eps)\mapsto S_\eps(x)$ is upper semicontinuous.

\medskip
\textbf{Semicontinuity of the marginal functions.}
Because $S_\eps$ is compact-valued and upper semicontinuous, Berge's
maximum theorem gives that
\[
(x,\eps)\mapsto \psi_\eps^p(x)=\max_{y\in S_\eps(x)}F(x,y)
\]
is upper semicontinuous and
\[
(x,\eps)\mapsto \psi_\eps^o(x)=\min_{y\in S_\eps(x)}F(x,y)
\]
is lower semicontinuous. Since $-\psi_\eps^o$ is upper
semicontinuous, their difference
\[
\Delta_\eps(x)=\psi_\eps^p(x)-\psi_\eps^o(x)
=\psi_\eps^p(x)+\bigl(-\psi_\eps^o(x)\bigr)
\]
is upper semicontinuous as the sum of two upper semicontinuous
functions.

\medskip
\textbf{Lower semicontinuity of $S_\eps$ at $(x_0,0)$.}
Fix $y_0\in S(x_0)$ and let $(x_k,\eps_k)\to(x_0,0)$.
Because $y_0\in S(x_0)$, $f(x_0,y_0)=\phi(x_0)$.
By continuity of $f$ and $\phi$, $f(x_k,y_0)-\phi(x_k)\to 0$.
For all sufficiently large $k$, we have $x_k\in U$, so the quadratic
growth condition yields
\[
f(x_k,y_0)-\phi(x_k)
\ge \underline\mu\,\dist\bigl(y_0,S(x_k)\bigr)^2.
\]
Therefore $\dist\bigl(y_0,S(x_k)\bigr)\to 0$.
Since $S(x_k)$ is nonempty and compact, there exists
$\tilde y_k\in S(x_k)$ such that
\[
\|\tilde y_k-y_0\|=\dist\bigl(y_0,S(x_k)\bigr).
\]
Hence $\tilde y_k\to y_0$. Moreover, $S(x_k)\subseteq S_{\eps_k}(x_k)$
because $\eps_k\ge 0$. Thus $\tilde y_k\in S_{\eps_k}(x_k)$ and
$\tilde y_k\to y_0$, which proves lower semicontinuity of
$S_\eps$ at $(x_0,0)$.

\medskip
\textbf{Lower semicontinuity of $S_\eps$ at $(x_0,\eps_0)$ for
$\eps_0>0$.}
Fix $y_0\in S_{\eps_0}(x_0)$ and let $(x_k,\eps_k)\to(x_0,\eps_0)$.

If
\[
f(x_0,y_0)<\phi(x_0)+\eps_0,
\]
then by continuity of $f$ and $\phi$ we have, for all sufficiently
large $k$, $f(x_k,y_0)\le \phi(x_k)+\eps_k$, so we may simply take $y_k:=y_0$.

Now suppose
\[
f(x_0,y_0)=\phi(x_0)+\eps_0.
\]
Let $\hat y\in S(x_0)$, which exists by Assumption~\ref{ass:standing}.
For any $\alpha\in(0,1)$, define
\[
y_\alpha:=(1-\alpha)y_0+\alpha \hat y.
\]
Since $Y$ is convex, $y_\alpha\in Y$. By convexity of
$f(x_0,\cdot)$ on $Y$,
\[
f(x_0,y_\alpha)
\le (1-\alpha)f(x_0,y_0)+\alpha f(x_0,\hat y)
= (1-\alpha)\bigl(\phi(x_0)+\eps_0\bigr)+\alpha\phi(x_0)
= \phi(x_0)+(1-\alpha)\eps_0
< \phi(x_0)+\eps_0.
\]
Fix a sequence $\alpha_m\downarrow 0$. For each $m$, continuity
again implies that there exists $K_m$ such that for all
$k\ge K_m$,
\[
f(x_k,y_{\alpha_m})\le \phi(x_k)+\eps_k,
\]
so $y_{\alpha_m}\in S_{\eps_k}(x_k)$. By choosing $K_m$
strictly increasing and defining $y_k:=y_{\alpha_m}$ for
$K_m\le k<K_{m+1}$, we obtain a sequence with
$y_k\in S_{\eps_k}(x_k)$ and $y_k\to y_0$. This proves lower
semicontinuity at $(x_0,\eps_0)$.

\medskip
Finally, when $S_\eps$ is both upper and lower semicontinuous at
$(x_0,\eps_0)$, Berge's maximum theorem implies that
$\psi_\eps^o$ and $\psi_\eps^p$ are continuous at
$(x_0,\eps_0)$, and therefore so is $\Delta_\eps=\psi_\eps^p-\psi_\eps^o$.
\end{proof}

\subsection*{Exactness of the Nikaido--Isoda penalization}
This is standard~\cite{fukushima1992equivalent,von2009optimization,harker1991generalized}
and is reproduced for self-containment. Fix $x\in X$. For any
$(y,v)\in Y\times Y$, $\mathcal N_x(y,v)\ge 0$, with equality iff
$(y,v)$ is a lower equilibrium of the pessimistic
GNEP~\eqref{eq:MF}. By compactness of $Y$ and continuity of the
integrands, the suprema defining $\mathcal N_x$ are attained.
Continuity of $\mathcal N_x$ follows under the standard
continuity assumptions for the associated feasible
correspondence; these are satisfied in the compact examples
considered below and are standard in NI-gap penalty
analyses~\cite{fukushima1992equivalent,von2009optimization}.
By compactness, for each $\sigma>0$, \eqref{eq:ni-penalty} admits
a minimizer $(y^\sigma,v^\sigma)$. Along any sequence
$\sigma_t\uparrow\infty$, $\{(y^{\sigma_t},v^{\sigma_t})\}$ is
bounded; pass to a convergent subsequence with limit
$(y^\star,v^\star)$.

\emph{Vanishing NI gap.}
Let $(\bar y,\bar v)$ be a lower equilibrium attaining
$\psi_\eps^p(x)$, which exists by compactness. Since
$\mathcal N_x(\bar y,\bar v)=0$, optimality of
$(y^{\sigma_t},v^{\sigma_t})$ for the penalized problem at
$\sigma_t$ gives
\begin{equation}\label{eq:ni-comp}
  -F(x,y^{\sigma_t})+\sigma_t\,\mathcal N_x(y^{\sigma_t},v^{\sigma_t})
  \le -F(x,\bar y).
\end{equation}
Rearranging~\eqref{eq:ni-comp},
\begin{equation}\label{eq:ni-bound}
  \sigma_t\,\mathcal N_x(y^{\sigma_t},v^{\sigma_t})
  \le F(x,y^{\sigma_t})-F(x,\bar y)
  \le \max_{y\in Y}F(x,y)-F(x,\bar y),
\end{equation}
so the left-hand side is bounded above by a finite constant
independent of $t$ (continuity of $F$ on the compact set $Y$
ensures the maximum is finite). With $\sigma_t\uparrow\infty$, the
inequality forces $\mathcal N_x(y^{\sigma_t},v^{\sigma_t})\to 0$
along the subsequence; by continuity of $\mathcal N_x$,
$\mathcal N_x(y^\star,v^\star)=0$, so $(y^\star,v^\star)$ is a
lower equilibrium of \eqref{eq:MF}.

\emph{Pessimistic optimality.}
Passing to the limit in \eqref{eq:ni-comp} and using
$\sigma_t\mathcal N_x(y^{\sigma_t},v^{\sigma_t})\ge 0$ together
with continuity of $F$ yields
\[
  -F(x,y^\star)\le -F(x,\bar y),
\]
i.e.\ $F(x,y^\star)\ge F(x,\bar y)=\psi_\eps^p(x)$. Conversely,
since $\mathcal N_x(y^\star,v^\star)=0$, the point $y^\star$ is an
admissible pessimistic lower-level response (i.e.\ feasible for
the inner pessimistic maximization defining $\psi_\eps^p(x)$), so
$F(x,y^\star)\le\psi_\eps^p(x)$. Combining the two inequalities,
$F(x,y^\star)=\psi_\eps^p(x)$.

\section{Reproducibility}\label{sec:repro}
All computations were performed in Python using NumPy, SciPy, and
Matplotlib. The supplementary code \texttt{reproduce.py} generates
all tables and figures reported in the paper. The fixed seed
\texttt{seed=7} is used for all Latin-hypercube samples.

\paragraph{Lower-level solves} use SciPy's \texttt{SLSQP} with
\texttt{ftol}$=10^{-12}$, \texttt{maxiter}$=500$.

\paragraph{Optimistic proximal scheme} uses $\tau=2.0$, max 120
iterations, tolerance $10^{-6}$ on the joint test.

\paragraph{Pessimistic NI-gap penalization} uses
$\sigma\in\{10,10^2,10^3,10^4\}$ with \texttt{SLSQP} inner solves
and multistart with $n_\mathrm{starts}=5$ in Case Study 1 and
$n_\mathrm{starts}=8$ in Case Study 2.

\paragraph{Frontier sweep} uses $J=21$ weights and
$N_\mathrm{LHS}\in\{80,120\}$ samples.

\paragraph{Status labels} follow \S\ref{sec:algo-frontier}:
\emph{converged} (stopping test satisfied), \emph{incumbent}
(iteration limit reached), \emph{heuristic} (hand-specified),
\emph{empirical Pareto} (undominated in sampled set).

\bibliographystyle{unsrt}
\bibliography{reference}

\end{document}